\documentclass[a4paper,12pt]{amsart}

\usepackage{a4wide}
\usepackage{tikz}
\usepackage{url}

\newif\ifdetails
\detailstrue
\newcommand{\DETAIL}[1]%
{\ifdetails\par\fbox{\begin{minipage}{0.9\linewidth}\textit{Detail:}
			#1\end{minipage}}\par\fi}
\newcommand{\TODO}[1]%
{\ifdetails\par\fbox{\begin{minipage}{0.9\linewidth}\textbf{TODO:}
			#1\end{minipage}}\par\fi}

\usepackage{makecell}
\usepackage{relsize}

\newtheorem{lem}{Lemma}
\newtheorem{pro}{Proposition}
\newtheorem{thm}{Theorem}
\newtheorem{corollary}{Corollary}

\theoremstyle{remark}

\newtheorem{question}{Question}

\newtheorem{problem}{Problem}

\usepackage{caption}
\usepackage{subcaption}
\usepackage{float} 
\usepackage[pdftex,a4paper,citecolor = blue, colorlinks=true,urlcolor=blue]{hyperref}
\newcommand{\old}[1]{{}}

\DeclareMathOperator{\X}{\boldsymbol{x}}
\DeclareMathOperator{\Y}{\boldsymbol{y}}
\DeclareMathOperator{\1}{\boldsymbol{1}}

\usetikzlibrary{arrows}

\title{The level matrix of a tree and its spectrum}

\author{Audace A. V. Dossou-Olory}

\address{Audace A. V. Dossou-Olory \\ D{\'e}partement d'Hydrologie et Gestion des Ressources en Eau \\ Institut National de l'Eau \\ Centre d'Excellence d'Afrique pour l'Eau et l'Assainissement  \\ and  Institut de Math\'ematiques et de Sciences Physiques, Dangbo \\ Universit\'e d'Abomey-Calavi,  B\'enin \\  \url{https://orcid.org/0000-0003-2065-117X} }
\email{audace@aims.ac.za}



\begin{document}

\begin{abstract}
Given a rooted tree $T$ with vertices $u_1,u_2,\ldots,u_n$, the level matrix $L(T)$ of $T$ is the $n \times n$ matrix for which the $(i,j)$-th entry is the absolute difference of the distances from the root to $v_i$ and $v_j$. This matrix was implicitly introduced by Balaji and Mahmoud~[{\em J. Appl. Prob.} 54 (2017) 701--709] as a way to capture the overall balance of a random class of rooted trees. In this paper, we present various bounds on the eigenvalues of $L(T)$ in terms of other tree parameters, and also determine the extremal structures among trees with a given order. Moreover, we establish bounds on the mutliplicity of any eigenvalue in the level spectrum and show that the bounds are best possible. Furthermore, we provide evidence that the level spectrum can characterise some trees. In particular, we provide an affirmative answer to a very recent conjecture on the level energy (sum of absolute values of eigenvalues).
\end{abstract}

\maketitle

\section{Introduction}

A common suitable means for storing graphs in computers, or applying mathematical methods to study their properties is the use of matrices to specify them. Spectral graph theory is concerned with the relationship between graph properties and the spectrum (set of all the eigenvalues) of matrices associated with graphs. It is known that no spectrum of a single matrix adequately describes all the facets of a graph. That is why there is generally a need to introduce new matrices. These include the adjacency matrix, the distance matrix, the signed and unsigned Laplacian, the maximum degree matrix~\cite{Adiga}, and the much recent notion of ancestral matrix~\cite{Eric}, hoping that combinations of these matrices can provide more structural information. The main aim is to reveal the properties of graphs that are characterised by the spectra of these matrices, and an extensive study has been conducted over the past five decades by several researchers.

\medskip
This paper focuses on the spectrum of the level matrix associated with a rooted tree. Given a rooted tree $T$ with vertices $u_1,u_2,\ldots,u_n$, the level matrix of $T$, denoted by $L(T)$, is the $n \times n$ matrix whose entry $l_{ij}$ in the $i$-th row and $j$-th column equals the absolute difference of the levels (distances from the root) of $u_i$ and $u_j$. We will study several properties of the eigenvalues of this matrix and place a particular attention to the spectral radius, which is the largest of the absolute values of the eigenvalues of a given matrix associated with a graph. We denote the level of a vertex $v$ in a rooted tree $T$ by $l(v)$. The absolute levels' difference $|l(v) -l(w)|$ is a way to measure how close vertices $v$ and $w$ are in a rooted tree, and it is worth pointing out that the level matrix is similar in nature to the usual distance matrix. In particular, the level matrix is symmetric and its diagonal entries are all equal to $0$.  So all eigenvalues of $L(T)$ are necessarily real numbers and they sum up to $0$, since the trace of $T$ is equal to $0$. For example, we depict in Figure~\ref{fig:standard-tree} a tree rooted at $v_1$ and its level matrix.
\begin{figure}[h]
  \begin{center}
\begin{tikzpicture}
  [scale=0.5,inner sep=0.8mm, 
   vertex/.style={circle,thick,draw}, 
   thickedge/.style={line width=2pt}] 
    \node[vertex] (a1) at (3,3) [fill=black] {};
    \node[vertex] (a2) at (1,2) [fill=white] {};
    \node[vertex] (a3) at (5,2) [fill=white] {};
    \node[vertex] (a4) at (0,1) [fill=white] {};
    \node[vertex] (a5) at (1,1) [fill=white] {};
    \node[vertex] (a6) at (5,1) [fill=white] {};
    \node[vertex] (a7) at (0,0) [fill=white] {};
    \node[vertex] (a8) at (4,0) [fill=white] {};
    \node[vertex] (a9) at (6,0) [fill=white] {}; 
   \draw[thick,black] (a1)--(a2)--(a4)--(a7);
   \draw[thick,black] (a1)--(a2)--(a5);
   \draw[thick,black] (a1)--(a3)--(a6)--(a8);
   \draw[thick,black] (a1)--(a3)--(a6)--(a9);   
   
     \node at (3,3.6) {$v_1$};
     \node at (1,2.6) {$v_6$};
     \node at (5,2.6) {$v_2$};
     \node at (-0.3,1.7) {$v_7$};
     \node at (1,0.3) {$v_5$};
     \node at (5.7,1) {$v_3$};
      \node at (0,-0.7) {$v_4$};
       \node at (4,-0.7) {$v_8$};
        \node at (6,-0.7) {$v_9$};
\end{tikzpicture} \qquad \qquad
$
  \begin{pmatrix}
		0 & 1 & 2 & 3 & 2 & 1 & 2 & 3 & 3\\
		1 & 0 & 1 & 2 & 1 & 0 & 1 & 2 & 2\\
		2 & 1 & 0 & 1 & 0 & 1 & 0 & 1 & 1\\
		3 & 2 & 1 & 0 & 1 & 2 & 1 & 0 & 0\\
		2 & 1 & 0 & 1 & 0 & 1 & 0 & 1 & 1\\
		1 & 0 & 1 & 2 & 1 & 0 & 1 & 2 & 2\\
		2 & 1 & 0 & 1 & 0 & 1 & 0 & 1 & 1\\
		3 & 2 & 1 & 0 & 1 & 2 & 1 & 0 & 0\\
		3 & 2 & 1 & 0 & 1 & 2 & 1 & 0 & 0
	\end{pmatrix}\,.$
\end{center}
\caption{A rooted tree and its level matrix (indexed by vertices $v_1,v_2,\ldots,v_9$ in this order).}\label{fig:standard-tree}
\end{figure}
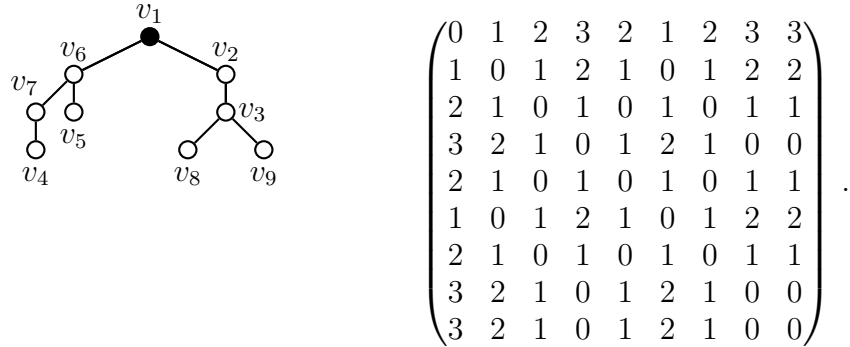

The level matrix has already found an application to applied probability as showed by Balaji and Mahmoud~\cite{Balaji}. The authors of~\cite{Balaji} employed the absolute levels' differences to define both what they called the \textit{level index} of a rooted tree and the \textit{Gini index} of a random class of rooted trees. The Gini index is an analogue of the standard Gini coefﬁcient commonly used in economics to study the distribution of income within a nation. This is also a motivation to study the properties of the level matrix of rooted trees. We will establish various bounds on the eigenvalues of the level matrix in terms of other tree parameters, and also describe the tree structures attaining the bounds. Moreover, we will show that the level spectrum can characterise the structure of some rooted trees. As a corollary to our results, we will determine the maximum level energy among trees with a given order, thus providing a positive answer to a very recent conjecture.

\section{Main results}

In analogy to other graph spectra, the level spectrum of a rooted tree $T$ is defined as the spectrum (set of all eigenvalues) of the level matrix $L(T)$.  In the example of Figure~\ref{fig:standard-tree}, the characteristic polynomial, $\det(xI_9 - L(T))$, is $x^9 - 80x^7 - 276x^6 - 216x^5$ and its zeros (the eigenvalues), rounded to few decimals, are given by:
$$ 10.415812724,~ 0,~ 0,~ 0,~ 0,~ 0,~ -1.1775860608,~ -2.6888645876,~ -6.5493620755\,.$$

We denote the maximum level of a vertex in a rooted tree $T$ by $l_{\max}(T)$. By $M^T$ we mean the transpose of a matrix $M$. A simple upper bound for an eigenvalue of a level matrix can be obtained using the maximum level of the vertices.

\begin{pro}
Let $T$ be a rooted tree with $n$ vertices. We have $$|\lambda| \leq (n-1) l_{max}(T)$$ for every eigenvalue $\lambda$ of $L(T)$. Equality holds for $n\leq 2$.
\end{pro}

\begin{proof}
Let $\X=(x_1~ x_2~ \ldots ~ x_n)^T$ be an eigenvector corresponding to an eigenvalue $\lambda$ of $L(T)$. By the definition of an eigenvalue, 
\begin{align*}
\lambda x_i =(l_{i1}~ l_{i2}~ \ldots ~ l_{in}) \X &=l_{i1}x_1+ l_{i2} x_2 +\cdots + l_{in}x_n\,,\\
|\lambda| \cdot | x_i| &\leq l_{i1}|x_1|+ l_{i2} |x_2| +\cdots + l_{in}|x_n|\\
&\leq (l_{i1} + l_{i2} +\cdots + l_{in} )\max_{1\leq j \leq n}|x_j|\,.
\end{align*}

Since $l_{ii}=0$ and $\max_{1\leq j \leq n}|x_j| \neq 0$ by definition of eigenvector, this implies that
\begin{align*}
|\lambda| \cdot | x_i| \leq (n-1)l_{\max}(T) \max_{1\leq j \leq n}|x_j|
\end{align*}
for all $1\leq i \leq n$. In particular, we obtain $|\lambda| \leq (n-1)l_{\max}(T)$.
\end{proof}

\medskip
A basic identity can be obtained for the eigenvalues of a level matrix as follows.

\begin{pro}\label{Prop:traceLcarre}
Let $T$ be a rooted tree with $n$ vertices and $\lambda_1,\lambda_2,\ldots,\lambda_n$ the eigenvalues of $L(T)$. We have
\begin{align*}
\lambda_1^2 + \lambda_2^2+ \cdots +\lambda_n^2= \sum_{i=1}^n (l_{i1}^2 + l_{i2}^2 + \cdots +l_{in}^2) =2 \sum_{1\leq i<j \leq n} l_{ij}^2\,.
\end{align*}
\end{pro}

\begin{proof}
By assumption, all the eigenvalues of $L(T)^2$ are $\lambda_1^2,\lambda_2^2,\ldots,\lambda_n^2$. Thus, the quantity $\lambda_1^2 + \lambda_2^2+ \cdots +\lambda_n^2$ corresponds to the trace de $L(T)^2$. By definition, the $(i,i)$-th entry of $L(T)^2$ is given by
$l_{i1}l_{1i}+l_{i2}l_{2i}+\cdots + l_{in}l_{ni}=l_{i1}^2 +l_{i2}^2 +\cdots + l_{in}^2$, from which we obtain
\begin{align*}
\lambda_1^2 + \lambda_2^2+ \cdots +\lambda_n^2= \sum_{i=1}^n (l_{i1}^2 + l_{i2}^2 + \cdots +l_{in}^2)\,.
\end{align*}
\end{proof}

In order to study further properties of the eigenvalues of the level matrix, we need to recall some important tools for linear algebra. For a symmetric matrix $M$ and a nonzero vector $\X$, the Rayleigh quotient, denoted by $R(M,\X)$, is given by
$$ R(M,\X)=\frac{\X^T M \X}{\X^T \X}\,.$$ It is well-known that $ R(M,\X)=\lambda$ if $\X$ is an eigenvector for the eigenvalue $\lambda$. Moreover, the smallest and the greatest eigenvalues of $M$ can be computed by:
\begin{align*}
\inf_{||\X||=1}  R(M,\X)=\inf_{||\X||=1}  \X^T M \X \quad \text{and} \quad \sup_{||\X||=1}  R(M,\X)=\sup_{||\X||=1}  \X^T M \X\,,
\end{align*}
respectively (infirmum and supremum taken over all unit vectors)~\cite{Bapat} (see also~\cite[p. 456]{West}).

\medskip
\subsection{Level spectral radius}

We continue our considerations with the level spectral radius, i.e. the largest absolute value of the eigenvalues of the level matrix of a rooted tree $T$, which we denote by $\rho_L(T)$.

\medskip
Invoking Proposition~\ref{Prop:traceLcarre} and using the fact that $\rho_L(T)^2$ is also the spectral radius of $L^2(T)$, we obtain:

\begin{corollary}
Let $T$ be a rooted tree with $n$ vertices. We have
\begin{align*}
\rho_L(T)^2 \geq  \frac{1}{n}\sum_{i=1}^n (l_{i1}^2 + l_{i2}^2 + \cdots +l_{in}^2) = \frac{2}{n} \sum_{1\leq i < j \leq n} l_{ij}^2\,.
\end{align*}
Equality holds for $n\leq 2$.
\end{corollary}

\medskip

\begin{thm}[Perron-Frobenius Theorem, \cite{Meyer}]
If a matrix $M$ is irreducible and has non-negative entries, then its spectral radius is an eigenvalue of multiplicity one. It corresponds to a unique positive unit eigenvector.
\end{thm}

The Perron-Frobenius Theorem is frequently used as a tool in computing the spectral radii of some of the matrices associated with graphs. In
particular, if a positive vector is an eigenvector then its corresponding positive eigenvalue is the spectral radius.

\begin{lem}
The level matrix $L(T)$ of a rooted tree $T$ is irreducible.
\end{lem}

\begin{proof}
Suppose that $L(T)$ is reducible. Then by the definition of reducible matrices, the vertex set of $T$ can be partitioned into
two non-empty subsets $U$ and $V$ such that $|l(u)-l(v)| = 0$ for all $u \in U$ and $v \in V$. In particular, $l(u)=l(r)=0$ for all $u \in U$, where $r\in V$ is the root of $T$. However, this is a contradiction to the fact that $r$ is the only vertex of $T$ whose level is $0$.
\end{proof}
 
Since $L(T)$ is an irreducible matrix, by the Perron-Frobenius Theorem, its level spectral radius is actually the largest positive eigenvalue and there exists an eigenvector with positive entries for $\rho_L(T)$. We call any such eigenvector a Perron vector of $L(T)$. 

\medskip
The level index of a rooted tree $T$, denoted by $LI(T)$, is defined as half the sum of all entries of $L(T)$~\cite{Balaji}. To simplify notation, let us set
$$L_i(T):=  \sum_{j=1}^n l_{ij}\,,~ \text{where}~ l_{ij}=|l_i-l_j|~ \text{and} ~ l_j=l(v_j)\,.$$

\begin{lem}\label{Lem:EqLB}
Let $T$ be a rooted tree with vertex set $\{v_1,v_2,\ldots,v_n\}$ such that $n>1$. Assume that
$l(v_1)\geq l(v_2)\geq \cdots \geq l(v_n)$. Then 
\begin{align*}
L_i - L_k=(n-2i)l_i  - 2 \sum_{j=i+1}^{k-1} l_j - (n-2k+2) l_k  
\end{align*}
holds for all $1\leq i < k \leq n$. In particular, 
\begin{align*}
L_i - L_k \geq (n-2k+2)(l_i -l_k)
\end{align*}
for all $1\leq i < k \leq n$, with equality if and only if $l_i=l_{i+1}=\cdots =l_{k-1}$.
\end{lem}
 
\begin{proof}
By definition, 
\begin{align*}
L_i =\sum_{j=1}^{i-1} (l_j-l_i) + \sum_{j=i+1}^{n} (l_i-l_j) = (n-2i+1)l_i+ \sum_{j=1}^{i-1} l_j - \sum_{j=i+1}^{n} l_j\,,
\end{align*}
which implies that
\begin{align*}
L_i &= (n-2k+1 +2(k-i) )l_i + \sum_{j=1}^{i-1} l_j - \sum_{j=i+1}^{k} l_j - \sum_{j=k+1}^{n} l_j\,,\\
L_k &= (n-2k+1 )l_k + \sum_{j=1}^{i-1} l_j + \sum_{j=i}^{k-1} l_j - \sum_{j=k+1}^{n} l_j\,,\\
L_i - L_k &=(n-2k+1)(l_i -l_k)+2(k-i)l_i - (l_i+l_k) - 2 \sum_{j=i+1}^{k-1} l_j
\end{align*}
for all $1\leq i < k \leq n$, where an empty sum is treated as $0$. Thus, the identity
\begin{align*}
L_i - L_k=(n-2i)l_i - (n-2k+2) l_k  - 2 \sum_{j=i+1}^{k-1} l_j
\end{align*}
follows. Furthermore, we have $\sum_{j=i+1}^{k-1} l_j \leq (k-i-1)l_i$ with equality if and only if $l_i=l_{i+1}=\cdots =l_{k-1}$. This implies that $L_i - L_k \geq (n-2k+2)(l_i -l_k)$.
\end{proof}

 \medskip

By $\1$ we mean a column vector whose entries are all equal to $1$. The lower bound in the following result relates the level index of a rooted tree to its level spectrum. 

\begin{thm}\label{Thm:FirstB}
Let $T$ be a $n$-vertex rooted tree. We have
\begin{align*}
 \frac{2}{n} LI(T)  \leq \rho_L(T) \leq \max_{1\leq i \leq n} L_i(T) \,.
\end{align*}
Equality holds in the lower bound if and only if $n\leq 2$.
\end{thm}

\begin{proof}
The statement of the theorem is trivial for $n\leq 2$. So we assume $n\geq 3$. By definition, the quantity $\1^T L(T) \1$ is the sum of all entries of $L(T)$. Using the Rayleigh quotient, we obtain
\begin{align}\label{Eq:Equality}
\rho_L(T) \geq R(L(T),\1 / \sqrt{n}) = 2LI(T) /n\,,
\end{align}
proving the lower bound. If equality holds, then by the Perron-Frobenius theorem, $\1 / \sqrt{n}$ is the unique positive unit eigenvector corresponding to $\rho_L(T)$. Hence $L(T) \1  =\rho_L(T) \1 $, i.e. the row sums of $L(T)$ are all equal to $\rho_L(T)$, which implies that $L_i=L_k$ for all $1\leq i< k \leq n$. Using Lemma~\ref{Lem:EqLB}, we obtain
\begin{align*}
L_{k-1} - L_k = (n-2k+2)(l_{k-1} -l_k)
\end{align*}
for all $2\leq k \leq n$. We choose (always possible) $k$ such that $l_{k-1} \neq l_k$. Also, recall that $l_{k-1} \geq l_{k}$. We consider two scenarios:
\begin{itemize}
\item $k-1 \neq n/2$. In this case, $L_{k-1} - L_k \neq 0$ which is a contradiction. 
\item $k-1 = n/2$. In this case, $n$ must be even and $k\geq 3$ is the sole index for which $l_{k-1}\neq l_k$. Thus, we have
$l_1=\cdots=l_{k-1} >l_k=\cdots=l_n=0$, which is absurd.
\end{itemize}
We conclude that equality never holds in~\eqref{Eq:Equality} for $n>2$.

\medskip 
For the upper bound, let $\Y=(y_1 ~y_2 ~\ldots~ y_n)^T$ be a Perron vector of $L(T)$. We have
\begin{align*}
\rho_L(T) y_i&=(l_{i1} ~ l_{i2} ~ \ldots ~l_{in}) \Y= l_{i1}y_1 + l_{i2}y_2 + \cdots +l_{in}y_n\\
& \leq (l_{i1} + l_{i2} + \cdots +l_{in}) \max_{1\leq j \leq n} y_j
\end{align*}
for all $1\leq i \leq n$. In particular, we get
$\rho_L(T) \leq l_{k1} + l_{k2} + \cdots +l_{kn} $
with $k$ being an index of the maximum value of the entries of $\Y$. Therefore, 
\begin{align*}
\rho_L(T) \leq \max_{1\leq i \leq n} \sum_{j=1}^n l_{ij}=\max_{1\leq i \leq n} L_i(T)
\end{align*}
follows immediately. 
\end{proof}

\medskip
Since $(\sum_{i=1}^n L_i/n)^2 \leq n \sum_{i=1}^n (L_i/n)^2$ by Cauchy-Schwartz inequality, we can get improvement for the lower bound stated in Theorem~\ref{Thm:FirstB}.

\begin{thm}
Let $T$ be an $n$-vertex rooted tree. Then it holds that
\begin{align*}
\rho_L(T) \geq \sqrt{\frac{1}{n}\sum_{j=1}^n L_j(T) ^2 }\,.
\end{align*}
\end{thm}

\begin{proof}
Recall that $\rho_L(T)^2$ is the spectral radius of $L(T)^2$. Let $\X$ be the unit Perron vector of $L(T)$. Then $\X$ is an eigenvector for the eigenvalue $\rho_L(T)^2$ and using the Rayleigh quotient, we obtain
\begin{align*}
\rho_L(T)^2 = \X^T L(T)^2 \X \geq \1^T L(T)^2 \1 /n = (\1^T L(T)) ( L(T) \1) /n\,.
\end{align*}
Moreover, $\1^T L(T)$ equals the vector $(L_1 ~ L_2~ \ldots ~ L_n)$ and $ L(T) \1 =(L_1 ~ L_2~ \ldots ~ L_n)^T$. Thus, we have
$(\1^T L(T)) ( L(T) \1) = (L_1^2 + L_2^2+ \cdots + L_n^2)$, which implies the inequality
\begin{align*}
\rho_L(T)^2 \geq \frac{1}{n}(L_1^2 + L_2^2+ \cdots + L_n^2)
\end{align*}
as stated.
\end{proof}

\medskip
Let us mention that it is still possible to improve the lower bound stated in the previous theorem. Taking the unit vector
\begin{align*}
\frac{1}{\sqrt{L_1^2 + L_2^2+ \cdots + L_n^2}} (L_1 ~ L_2~ \ldots ~ L_n)
\end{align*}
instead of $\1 /\sqrt{n}$ in the proof of the previous result, we can establish the following.

\begin{thm}
Let $T$ be an $n$-vertex rooted tree. Then 
\begin{align*}
\rho_L(T) \geq \sqrt{\frac{\sum_{1\leq i \leq n} Q_i^2}{\sum_{1\leq j \leq n} L_j^2}} =\sqrt{\frac{\sum_{1\leq i \leq n} Q_i^2}{\sum_{1\leq i \leq n} Q_i}} 
\end{align*}
holds, where $Q_i=\sum_{1\leq j \leq n} l_{ij} L_j$.
\end{thm}
The details are omitted.

\medskip
We define a rooted star to be a star whose root is the central vertex, and by rooted path we mean a path rooted at one of its endvertices.

\begin{thm}
Among all rooted trees with $n$ vertices, the rooted star $S_n$ uniquely minimises the level spectral radius. Moreover, $\rho_{L}(S_n)=\sqrt{n-1}$.
\end{thm}

\begin{proof}
Let $T$ be an $n$-vertex rooted tree and $\Y=(y_1~ y_2 ~ \ldots ~ y_n)^T$ the unit Perron vector of $L(S_n)$. We have
\begin{align*}
R(L(T), \Y) - R(L(S_n), \Y)&=\Y^T L(T) \Y - \Y^T L(S_n) \Y\\
& = \sum_{1\leq j \leq n} \sum_{1\leq i \leq n} \big(l_{ij}(T)- l_{ij}(S_n)\big)y_i y_j\,.
\end{align*}
Note that $l_{ij}(T) \geq 0$ for all non-root vertices of $T$ and that $l_{ij}(T) \geq 1$ if one of these vertices (not both) is the root of $T$. Thus, we have $l_{ij}(T) \geq l_{ij}(S_n)$ for a suitable ordering of the vertices, for all $1\leq i,j \leq n$. Equality holds if and only if all non-root vertices of $T$ have level $1$, i.e. if $T$ is the rooted star $S_n$. Therefore, $R(L(T), \Y) - R(L(S_n), \Y)\geq 0$ and
\begin{align*}
\rho_{L}(T)\geq R(L(T), \Y) > R(L(S_n), \Y)=\rho_{L}(S_n)
\end{align*}
for $T\neq S_n$. It can be noted that $L(S_n)$ is the same as the adjacency matrix of $S_n$. Hence, we have $\rho_{L}(S_n)=\sqrt{n-1}$.
\end{proof}

\medskip
The matrix comparison argument used in the previous proof leads us to the following observation.

\begin{lem}\label{Lem:landd}
For a rooted tree $T$, we have $l_{ij}\leq d_{ij}$, where $d_{ij}$ is the $(i,j)$-th entry of the distance matrix of $T$. 
	\end{lem}
	
	\begin{proof}
	Consider two vertices $v_i$ and $v_j$ of $T$. Let $u$ be the last vertex on the common subpath from the root (possibly, $u$ can coincide with the root of $T$) to $v_i$ and $v_j$. We have $$l(v_i)-l(v_j)= d(v_i,u) - d(v_j,u)\,.$$
	
	If $v_i$ and $v_j$ lie on the same path to the root of $T$, then one of these vertices, say $v_j$ coincides with $u$, implying that
	$d_{ij}=d(v_i,v_j)=d(v_i,u)$. Thus, $l(v_i)-l(v_j)= d_{ij}$ holds.
	
	If $v_i$ and $v_j$ do not lie on the same path to the root of $T$, then none of these vertices coincides with $u$, implying that
	$d_{ij}=d(v_i,v_j)=d(v_i,u)+d(u,v_j)$. Thus, we obtain $$l(v_i)-l(v_j)= d(v_i,u) - d(v_j,u) <d(v_i,u)+d(u,v_j)=d_{ij}\,.$$
	\end{proof}

\medskip
For a tree $T$, we denote by $D(T)$ its distance matrix.

\begin{thm}\label{Thm:Pathmaxim}
Among all rooted trees with $n$ vertices, the rooted path $P_n$ uniquely maximises the level spectral radius. Moreover,
$\rho_{L}(P_n)=1/ (\cosh(t)-1)$, where $t>0$ satisfies 
$$\tanh(t/2) \tanh(n\cdot t/2)=1/n\,.$$
\end{thm}

\begin{proof}
Let $T$ be an $n$-vertex rooted tree and $\Y=(y_1~ y_2 ~ \ldots ~ y_n)^T$ the unit Perron vector of the level matrix $L(T)$. We have
\begin{align*}
R(D(T), \Y) - R(L(T), \Y)& = \sum_{1\leq i \leq n} \sum_{1\leq j \leq n}(d_{ij} - l_{ij})y_i y_j \geq 0\,,
\end{align*}
where the inequality follows from Lemma~\ref{Lem:landd}. Consequently, we get
\begin{align*}
R(D(T), \Y) \geq R(L(T), \Y) = \rho_{L}(T)
\end{align*}
with equality if and only if $d_{ij} = l_{ij}$ for all $i,j$, in which case $T$ must coincide with the rooted path $P_n$ (see the proof of Lemma~\ref{Lem:landd}). In~\cite{Ruzieh} Ruzieh and Powers proved that the $n$-vertex path has the maximum distance spectral radius among all $n$-vertex trees. Using this result, we obtain
 \begin{align*}
\rho_{L}(T) \leq R(D(T), \Y) \leq \rho_{D}(T) \leq \rho_{D}(P_n)\,.
\end{align*}
The formula for $\rho_L(P_n)=\rho_D(P_n)$ can also be found in~\cite{Ruzieh}, thus completing the proof.
\end{proof}

\medskip
Another bound can be obtained for the level spectral using a quotient matrix of $L(T)$. This is shown in the next theorem.

\begin{thm}
Let $T$ be a rooted tree with $n>1$ vertices. We have
\begin{align*}
\rho_L(T) \geq \frac{1}{n-1} \max_{1\leq i \leq n} \Big(LI - L_i + \sqrt{(LI- L_i)^2 +(n-1)L_i^2}  \Big)\,.
\end{align*}
\end{thm}

\begin{proof}
Let $T$ be a rooted tree with vertex set $\{v_1,v_2,\ldots,v_n \}$. Fix $1\leq i \leq n$ and reorder the vertices of $T$ starting from $v_i$ to construct $L(T)$. Then $L(T)$ can be partitionned into blocks/submatrices $A_{1,1}, A_{1,2}, A_{2,1}, A_{2,2}$ as follows:
\begin{align*}
\begin{array}{c|c|ccc cccc|}
 & v_i & v_1  &\cdots &v_{i-1} &  v_{i+1}& v_{i+2}&\cdots &v_n \\  \hline  
v_i & A_{1,1}=0 &  & & &  A_{1,2} & & & \\ \hline
v_1 &   & & & & & & & \\ 
\vdots &   & & & & & & & \\ 
v_{i-1} &   & & &   &  & & &\\ 
 v_{i+1}  &  A_{2,1} & & & & A_{2,2} & & &  \\ 
v_{i+2} &   & & & & & & &  \\ 
 \vdots &    & & & & & & & \\  
 v_n &     & & & &  & & & \\ \hline
\end{array}
\end{align*} 
The average row sum of $A_{1,1}, A_{1,2}, A_{2,1}, A_{2,2}$ gives $$0,~ L_i, L_i/(n-1), 2(LI-L_i)/ (n-1)\,,$$ respectively. Then the quotient matrix corresponding to this partition of $L(T)$ is given by
\begin{align*}
Q(T):=\begin{pmatrix}
	0 & L_i\\
	 L_i/(n-1) & 2(LI-L_i)/ (n-1)		
	\end{pmatrix} \,. 
	\end{align*}
The two eigenvalues of $Q(T)$ are solutions of the quadratic equation in variable $X$:
\begin{align*}
X^2 -  2 X (LI-L_i)/ (n-1) - L_i^2/(n-1)=0\,,
\end{align*}
of which the largest one is given by
\begin{align*}
\frac{1}{n-1}\Big(LI - L_i + \sqrt{(LI- L_i)^2 +(n-1)L_i^2}  \Big)\,.
\end{align*}
It is known (see~\cite{Cauchy}) that the eigenvalues of $Q(T)$ interlace those of $L(T)$, and the result follows.
\end{proof}

\medskip
To simplifly notation, let us set
\begin{align*}
H(T):= \sum_{i=1}^n (l_{i1}^2 + l_{i2}^2 + \cdots +l_{in}^2)=2\sum_{1\leq i<j \leq n} l_{ij}^2
\end{align*}
for a rooted tree $T$ with level matrix $L(T)=(l_{ij})_{1\leq i,j \leq n}$. We can obtain further bounds for the eigenvalues of the level matrix as follows.

\begin{thm}
Let $T$ be a $n$-vertex rooted tree and $\lambda$ an eigenvalue of $L(T)$. Then it holds that
\begin{align*}
\lambda^2 \leq \frac{n-1}{n}H(T)\,.
\end{align*}
\end{thm}

\begin{proof}
Let $\lambda_1, \lambda_2, \ldots, \lambda_n$ be all the eigenvalues of $L(T)$. Since $\sum_{i=1}^n \lambda_i =0$, we have
\begin{align*}
|\lambda_j| = \Big| \sum_{i=1, i\neq j}^n \lambda_i  \Big| \leq \sum_{i=1, i\neq j}^n |\lambda_i |\,.
\end{align*}
The Cauchy-Schwartz inequality yields
\begin{align*}
\Big( \sum_{i=1, i\neq j}^n |\lambda_i| \Big)^2 \leq (n-1) \sum_{i=1, i\neq j}^n \lambda_i^2\,,
\end{align*}
which implies 
\begin{align*}
\lambda_j^2  \leq \Big( \sum_{i=1, i\neq j}^n |\lambda_i | \Big)^2 \leq (n-1) \sum_{i=1, i\neq j}^n \lambda_i^2 \,.
\end{align*}
Furthermore, using the identity $H(T)=\sum_{i=1}^n \lambda_i^2$ established in Proposition~\ref{Prop:traceLcarre}, we obtain
\begin{align*}
\lambda_j^2   \leq (n-1) \sum_{i=1, i\neq j}^n \lambda_i^2 =(n-1) (H(T)- \lambda_j^2)\,,
\end{align*}
or equivalently, $n \cdot \lambda_j^2 \leq (n-1) H(T)$. This completes the proof.
\end{proof}

\medskip
The previous bound on $\lambda$ can be improved further as shown in the next theorem.

\begin{thm}\label{Thm:Lupas}
Let $T$ be a rooted tree with $n>2$ vertices and $\lambda_1 \geq \lambda_2 \geq \cdots \geq \lambda_n$ all the eigenvalues of $L(T)$. Then the following inequalities hold:
\begin{align*}
\sqrt{\frac{H(T)}{n(n-1)}}\leq ~\lambda_1 \leq  \sqrt{\frac{n-1}{n}H(T)} \,, \quad 
- \sqrt{\frac{n-1}{n} H(T) } \leq ~\lambda_n  \leq - \sqrt{\frac{H(T)}{n(n-1)}} \,,
\end{align*}
and
\begin{align*}
 - \sqrt{\frac{(j-1) H(T) }{n(n-j+1)}} \leq ~ \lambda_j \leq  \sqrt{\frac{(n-j) H(T) }{j\cdot n}}
\end{align*}
for all $j\in \{2,\ldots,n-1 \}$.
\end{thm}

\medskip
To prove Theorem~\ref{Thm:Lupas}, we will employ certain inequalities provided by A.~Lupas~\cite{Lupas}.

\begin{lem}[Theorem~2, \cite{Lupas}]\label{Lem:Lupas} 
Let $n>2$ be an integer and $P_n(x) \in \mathbb{R}[x]$ a monic polynomial of degree $n$ and only with real roots. If $x_1 \geq x_2 \geq \cdots \geq x_n$ are all the roots of $P_n(x)$, then the following relations hold:
\begin{align*}
x_1 \in \Big[y+\frac{1}{n}\sqrt{\frac{z}{n-1}},~ y +\frac{1}{n} \sqrt{(n-1)z} \Big]\,, \quad x_n  \in \Big[y-\frac{1}{n}\sqrt{(n-1)z},~ y-\frac{1}{n} \sqrt{\frac{z}{n-1}}\Big]\,,
\end{align*}
and
\begin{align*}
x_j \in \Big[y-\frac{1}{n}\sqrt{\frac{(j-1)z}{n-j+1}},~  y+\frac{1}{n}\sqrt{\frac{(n-j)z}{j}}\Big]
\end{align*}
for all $j\in \{2,\ldots,n-1 \}$, where
$$y=(x_1+x_2+\cdots +x_n)/n \quad \text{and} \quad z=n(x_1^2+x_2^2+\cdots +x_n^2) - (x_1+x_2+\cdots +x_n)^2\,.$$
\end{lem}

\medskip
\begin{proof}[Proof of Theorem~\ref{Thm:Lupas}]
We apply Lemma~\ref{Lem:Lupas} to the polynomial $$P_n(x)=(x-\lambda_1) (x-\lambda_2) \ldots (x- \lambda_n)\,.$$ Since $\lambda +\lambda_2+\cdots +\lambda_n =0$ and $\lambda_1^2+\lambda_2^2+\cdots +\lambda_n^2 =H(T)$ by Proposition~\ref{Prop:traceLcarre}, we obtain
\begin{align*}
\lambda_1 \in \Big[\frac{1}{n}\sqrt{\frac{n H(T) }{n-1}},~ \frac{1}{n} \sqrt{(n-1)n H(T) } \Big]\,,\quad \lambda_n \in \Big[-\frac{1}{n}\sqrt{(n-1)n H(T)},~ -\frac{1}{n} \sqrt{\frac{n H(T)}{n-1}}\Big]\,,
\end{align*}
and
\begin{align*}
\lambda_j  \in \Big[-\frac{1}{n}\sqrt{\frac{(j-1)n H(T) }{n-j+1}},~  \frac{1}{n}\sqrt{\frac{(n-j)n H(T) }{j}}\Big]
\end{align*}
for all $j\in \{2,\ldots,n-1 \}$. This completes the proof.
\end{proof}

\medskip
In particular, Theorem~\ref{Thm:Lupas} shows that the interval
\begin{align}\label{Equ:Lupas}
\Big[ - \sqrt{\frac{n-1}{n} H(T)},~ \sqrt{\frac{n-1}{n} H(T)}  \Big]
\end{align}
contains the spectrum of $L(T)$ for any rooted tree $T$ with $n$ vertices.

\medskip
\subsection{Level energy and eigenvalue's multiplicity}\label{Subs:multipli}

The energy of a graph, defined by Ivan Gutman~\cite{Gutman} in 1978, is a much studied quantity in the mathematical literature. In a manner fully analogous to other matrices associated with a graph, the level energy $E_L(T)$ of a rooted tree $T$ is the sum of the absolute values of the eigenvalues of $L(T)$~\cite{Audace}. We can apply Lemma~\ref{Lem:Lupas} to the polynomial
$$(x-|\lambda_1|) (x-|\lambda_2|)\ldots (x-|\lambda_n|)$$ to establish bounds for the level energy $E_L(T)$ similar to those in Theorem~\ref{Thm:Lupas}.

Since the trace of the level matrix of $T$ equals $0$, the level energy $E_L(T)$ is precisely twice the sum of those positive eigenvalues. In particular, we have $E_L(T) \geq 2 \rho_L(T)$ and any lower bound for the level spectral radius implies a lower bound for the level energy. Moreover, $E_L(T) = 2 \rho_L(T)$ if and only if $L(T)$ has precisely one positive eigenvalue. Thus, a natural question is to characterise rooted trees with only one positive level eigenvalue. In the sequel, we first manage some particular cases of this problem.

\medskip
The Cauchy interlacing theorem is quite often a key result in the approach of studying the spectrum of matrices associated with graphs. It captures the relationship between the spectrum of a symmetric matrix and that of its principal submatrices, see~\cite{Bapat2,Cauchy,Parlett}.

\begin{thm}\label{Thm:Mult0}
Let $T$ be a rooted tree with $n >1$ vertices and maximum level $l_{\max}$. Then the multiplicity of $0$ as an eigenvalue for $L(T)$ is at most $n-1-l_{\max}$. If equality holds, then $L(T)$ must have only one positive eigenvalue. In particular, equality holds for all rooted paths and all rooted versions of stars.
\end{thm}

\begin{proof}
Let $T$ be a tree with $n >1$ vertices and with root $r$. Denote by
$\lambda_1 \geq \lambda_2 \geq \cdots \geq \lambda_n$ all the eigenvalues of $L(T)$. Consider a subtree rooted at $r$ of $T$ which is a path $P_m$. We choose $m=1+l_{\max}\leq n$, which is always possible. Let $M$ be the submatrix of $L(T)$ obtained by deleting all rows and columns not indexed by the vertices of $P_m$. Note that the deletion of these vertices does not affect the levels of the other vertices in $T$. This means that $M$ is the level matrix of $P_m$. Let the eigenvalues of $M$ be $\mu_1 \geq \mu_2 \geq \cdots \geq \mu_m$. Applying the Cauchy's interlacing theorem~\cite{Bapat2, Cauchy}, we get
\begin{align*}
\lambda_{n-m+2} \leq \mu_2\,,~ \lambda_{n-m+3} \leq \mu_3\,,~ \ldots,~ ~\lambda_{n-1} \leq \mu_{m-1},~ \lambda_n \leq \mu_m\,.
\end{align*}
Since $L(P_m)=D(P_m)$ has precisely $m-1$ negative eigenvalues~\cite{Bapat2, Ruzieh}, it follows that
\begin{align*}
\lambda_{n-m+2}, \lambda_{n-m+3}, \ldots, \lambda_{n-1}, \lambda_n < 0\,.
\end{align*}
Thus, the number of nonnegative eigenvalues of $L(T)$ is at most $n-m+1$. In partciular, the multiplicity of $0$ as an eigenvalue for $L(T)$ is at most $n-m=n-1-l_{\max}$. If equality holds, then $L(T)$ must have only one positive eigenvalue.

The level matrix of the rooted star $S_n$ is the same as the adjacency matrix of $S_n$. Then it has only two nonzero eigenvalues given by $\pm \sqrt{n-1}$. Since the maximum level of $S_n$ is $1$, we see that the multiplicity of $0$ as an eigenvalue for $L(S_n)$ is precisely $n-2=n-1-l_{\max}$.

Let $R_n$ stand for the $n$-vertex star rooted at one of its non-central vertices. It was shown in~\cite{Audace} that $L(R_n)$ has precisely three nonzero eigenvalues, namely the zeros of $x^3+(-5n+9)x-4n+8$. Since the maximum level of $R_n$ is $2$, we see that the multiplicity of $0$ as an eigenvalue for $L(R_n)$ is precisely $n-3=n-1-l_{\max}$.

Recall that rooted paths are the only trees for which $0$ is not a level eigenvalue (see also~\cite{Audace}). This fact is further confirmed by this theorem since $l_{\max}=n-1$ for the rooted path $P_n$. 
\end{proof}

\medskip
Theorem~\ref{Thm:charcONE} below shows that besides rooted paths, the level spectrum can also specify the topological structure of other trees.

\begin{thm}\label{Thm:charcONE}
Let $n>2$ and $B$ be a $n\times n$ level matrix. Then $0$ is an eigenvalue of multiplicity $n-2$ for $B$ if and only if $B$ is associated with the rooted star $S_n$.
\end{thm}

\begin{proof}
We already know that $L(S_n)$ has only two nonzero eigenvalues. 

Conversely, let $T$ be a tree with $n >2$ vertices and with root $r$ such that $0$ is an eigenvalue of multiplicity $n-2$ for $L(T)$. Suppose (to the contrary) that $T$ is not the rooted star $S_n$. Denote by
$\lambda_1 > \lambda_2 = \cdots = \lambda_{n-1} > \lambda_n$ all the eigenvalues of $L(T)$. Consider a subtree rooted at $r$ of $T$ which is a path $P_m$. We choose $m\geq 3$, which is always possible since $T \neq S_n$. Let $M$ be the submatrix of $L(T)$ obtained by deleting all rows and columns not indexed by the vertices of $P_m$. Then $M$ is the level matrix of $P_m$ and its eigenvalues can be denoted by $\mu_1 \geq \mu_2 \geq \cdots \geq \mu_m$. By the Cauchy's interlacing theorem~\cite{Bapat2, Cauchy}, we have
\begin{align*}
\lambda_{m-1} \geq \mu_{m-1} \geq \lambda_{n-1} \,.
\end{align*}
Since $\lambda_{n-1}=0$ by assumption, we obtain $\mu_{m-1} \geq 0$. This is a contradiction for $L(P_m)=D(P_m)$ since $m\geq 3$. Hence $T$ must be the rooted star $S_n$.
\end{proof}

\medskip
We finish this paper by providing a result on the multiplicity of an arbitrary (not only $0$) eigenvalue in the spectrum of the level matrix of a tree. For a rooted tree $T$ and a leaf (pendant vertex) $v$ of $T$, let us denote by $\text{mul}_T(\lambda)$ the multiplicity of $\lambda$ as an eigenvalue of $L(T)$ and by $T-v$ the subtree obtained from $T$ by deleting $v$.

\begin{thm}\label{Thm:DiffMult}
Let $T$ be a rooted tree and $\lambda$ any eigenvalue of $L(T)$. Then
$$|\text{mul}_T(\lambda) - \text{mul}_{T-v}(\lambda) | \leq 1$$ holds for all leaves $v$ of $T$.
\end{thm}

\begin{proof}
The statement is trivial for $n=2$. So we assume $n>2$. Let $T$ be a rooted tree with $n$ vertices and $v$ a leaf of $T$. Denote by $\lambda_1 > \lambda_2 \geq \cdots \geq \lambda_n$ and $\mu_1 > \mu_2 \geq \cdots \geq \mu_{n-1}$ all the level eigenvalues of $T$ and $T-v$, respectively. Assume that $\lambda=\lambda_{k+1}$ has multiplicity $m$, i.e.
$$\lambda_1 \geq \cdots \geq \lambda_{k} > \lambda_{k+1}=\cdots = \lambda_{k+m}> \lambda_{k+m+1} \geq \cdots \geq \lambda_n\,.$$ By the Cauchy interlacing theorem~\cite{Bapat2, Cauchy}, we have
\begin{align*}
\lambda_{k} \geq \mu_k \geq  \lambda_{k+1}=\mu_{k+1}=\cdots =\mu_{k+m-1}= \lambda_{k+m}\geq \mu_{k+m}\geq \lambda_{k+m+1}\,.
\end{align*}
We observe four possible scenarios:
\begin{itemize}
\item If $\mu_k =  \lambda_{k+1}$ and $\lambda_{k+m}=\mu_{k+m}$, then $\lambda_{k} > \mu_k$ and $ \mu_{k+m} > \lambda_{k+m+1}$. In this case, we get $$\mu_{k-1}\geq \lambda_{k} > \mu_k=\cdots =\mu_{k+m} > \lambda_{k+m+1}\geq \mu_{k+m+1}\,.$$ Thus $\lambda=\lambda_{k+1}=\mu_k $ is an eigenvalue of multiplicity $m+1$ for $L(T-v)$. Hence $\text{mul}_T(\lambda) -  \text{mul}_{T-v}(\lambda)=-1$.
\item If $\mu_k >  \lambda_{k+1}$ and $\lambda_{k+m} > \mu_{k+m}$, then 
$$\mu_k >  \lambda_{k+1}= \mu_{k+1}=\cdots =\mu_{k+m-1} = \lambda_{k+m} >\mu_{k+m}\,,$$ which implies that $\lambda=\lambda_{k+1}=\mu_{k+1} $ is an eigenvalue of multiplicity $m-1$ for $L(T-v)$. Hence $\text{mul}_T(\lambda) -  \text{mul}_{T-v}(\lambda)= 1$.
\item If $\mu_k =  \lambda_{k+1}$ and $\lambda_{k+m} > \mu_{k+m}$, then $\lambda_{k} > \mu_k$ and we obtain
$$\mu_{k-1}\geq \lambda_{k} > \mu_k=\cdots =\mu_{k+m-1} > \mu_{k+m}\,.$$ Hence $\lambda=\lambda_{k+1}=\mu_{k} $ is an eigenvalue of multiplicity $m$ for $L(T-v)$, i.e. $\text{mul}_T(\lambda) -  \text{mul}_{T-v}(\lambda)= 0$.
\item If $\mu_k >  \lambda_{k+1}$ and $\lambda_{k+m} = \mu_{k+m}$, then $ \mu_{k+m} > \lambda_{k+m+1}$ and we get
$$\mu_k > \lambda_{k+1}=\mu_{k+1}=\cdots = \mu_{k+m} > \lambda_{k+m+1} \geq \mu_{k+m+1}\,.$$ Thus $\lambda=\lambda_{k+1}=\mu_{k+1} $ is an eigenvalue of multiplicity $m$ for $L(T-v)$, i.e. $\text{mul}_T(\lambda) -  \text{mul}_{T-v}(\lambda)= 0$.
\end{itemize}
This completes the proof.
\end{proof}

\medskip

For the special case of level eigenvalue $0$, Theorem~\ref{Thm:DiffMult} can be strengthened as follows.

\begin{thm}\label{Thm:zeroMult}
Let $T$ be a rooted tree with $n >2$ vertices and maximum level $l_{\max}$. Then the multiplicity of $0$ as an eigenvalue for $L(T)$ is precisely $n-1-l_{\max}$. Moroever, $L(T)$ always admits only one positive eigenvalue. Furthemore, 
$$\text{mul}_T(0) - \text{mul}_{T-v}(0) \in \{0,1\}$$ holds for all leaves $v$ of $T$.
\end{thm}

\begin{proof}
Let $\mathcal{S}_j$ represent the set of those vertices of $T$ having level $j$. For every $0\leq j \leq l_{\max}$, the rows of $L(T)$ indexed by the elements of $\mathcal{S}_j$ are all identical. Hence the rank of $L(T)$ is at most
$$n- \sum_{j=0}^{l_{\max}} (|\mathcal{S}_j| -1)=1+l_{\max}\,. $$
With this, $L(T)$ has $0$ as an eigenvalue with multiplicity at least $n-1- l_{\max}$. Invoking Theorem~\ref{Thm:Mult0} finishes the first part of the proof.

\medskip
Let $v$ be a leaf of $T$. If $v$ is not on the maximum level in $T$, or if there are two vertices (including $v$) on the maximum level in $T$, then the deletion of $v$ does not affect the maximum level, i.e. $l_{\max} (T-v)=l_{\max} (T)$. If $v$ is the only vertex lying on the maximum level in $T$, then the deletion of $v$ decreases the maximum level by exactly $1$, i.e. $l_{\max} (T-v)=l_{\max} (T) -1$. Therefore, 
$l_{\max} (T)  - l_{\max} (T-v)\in \{0,1\}$ holds. Now using the first part of the theorem, we obtain
\begin{align*}
\text{mul}_T(0) - \text{mul}_{T-v}(0)&=(n-1 - l_{\max} (T)) - (n-2-l_{\max} (T-v))\\
&=1- (l_{\max} (T)-l_{\max} (T-v))\,,
\end{align*}
which implies that $\text{mul}_T(0) - \text{mul}_{T-v}(0) \in \{0,1\}$.
\end{proof}

As a consequence of Theorem~\ref{Thm:zeroMult} and the starting discussion of Subsection~\ref{Subs:multipli}, we obtain an identity between the level energy and the level spectral.

\begin{corollary}\label{Coro:EnergySpectr}
For a rooted tree $T$, it holds that $E_L (T ) = 2 \rho_L (T )$.
\end{corollary}

\medskip
In particular, we have proved a very recent conjecture from~\cite{Audace}.

\begin{thm}
Among all rooted trees with $n$ vertices, the rooted path $P_n$ uniquely maximises the level energy.
\end{thm}

\begin{proof}
Simply combine Theorem~\ref{Thm:Pathmaxim} with Corollary~\ref{Coro:EnergySpectr}.
\end{proof}

\medskip
The following result appears in~\cite{Audace}:
\begin{align}\label{Eq:better}
E_L(T) \leq \sqrt{2n \sum_{1\leq i < j \leq n} l_{ij}^2}
\end{align}
holds for any rooted tree $T$ with $n$ vertices. We provide a new upper bound that is better than~\eqref{Eq:better}, as shown in Theorem~\ref{Tm:BetterThan6} below.

\begin{thm}\label{Tm:BetterThan6}
For a $n$-vertex rooted tree $T$ different from a rooted path, it holds that
\begin{align*}
E_L(T) \leq \sqrt{2(n-1) \sum_{1\leq i < j \leq n} l_{ij}^2}\,.
\end{align*}
\end{thm}

\begin{proof}

Let $x_1,x_2,\ldots,x_n$ be non-negative numbers and
$$X=\frac{1}{n} \sum_{i=1}^n x_i - \Big(\prod_{i=1}^n x_i \Big)^{1/n}\,.$$ It is shown in~\cite{Kober, Zhou} that
\begin{align*}
nX \leq n \sum_{i=1}^n x_i - \Big(\sum_{i=1}^n \sqrt{x_i}  \Big)^2 \leq n(n-1)X\,.
\end{align*}
Using the first inequality with $x_i=\lambda_i^2$, where $\lambda_1, \lambda_2, \ldots, \lambda_n$ denote all the eigenvalues of $L(T)$, we obtain
\begin{align*}
nX \leq n \sum_{i=1}^n \lambda_i^2 - \Big(\sum_{i=1}^n |\lambda_i|  \Big)^2 \,.
\end{align*}
Equivalently, 
\begin{align*}
&nX \leq 2n \sum_{1\leq i<j \leq n} l_{ij}^2 - E_L(T)^2 \,,\\
&X=\frac{1}{n} \sum_{i=1}^n \lambda_i^2  - \Big(\prod_{i=1}^n \lambda_i^2 \Big)^{1/n} =  \frac{2}{n} \sum_{1\leq i<j \leq n} l_{ij}^2 - |\det(L(T))|^{2/n} 
\end{align*}
by virtue of Proposition~\ref{Prop:traceLcarre}. Since $T$ is different from a rooted path, we have $\det(L(T))=0$. Therefore,
\begin{align*}
E_L(T)^2 \leq 2n \sum_{1\leq i<j \leq n} l_{ij}^2  -n X = 2n \sum_{1\leq i<j \leq n} l_{ij}^2  -2 \sum_{1\leq i<j \leq n} l_{ij}^2 \,.
\end{align*}
This completes the proof of the theorem.
\end{proof}

\medskip
\section{Concluding comments}

Recall that for the distance matrix, the spectrum of any tree contains exactly one positive eigenvalue~\cite[p. 104]{Bapat2}, and this is also the case for the level matrix of a rooted tree as shown in Theorem~\ref{Thm:zeroMult}. The level matrix and the distance matrix coincide only for rooted paths~\cite{Audace} (see also the proof of Lemma~\ref{Lem:landd}). On the other hand, $0$ is always an eigenvalue of $L(T)$ provided $T$ is not a rooted path, while this is not the case for the distance matrix. Since the main aim of spectral graph theory is to reveal the properties of graphs that are characterised by the spectrum of a matrix associated with it, we therefore propose some questions and problems, which we think may be interesting.

\medskip
The problem that characterises graphs having a number of distinct eigenvalues has been paid much attention for some matrices associated with graphs. In our current context, the level spectrum of any rooted tree which is not a rooted path contains at least two distinct nonzero eigenvalues. In Theorem~\ref{Thm:charcONE} we have shown that the rooted star $S_n$ is the only tree having precisely two distinct nonzero level eigenvalues. 

\begin{question}\label{Ques:distinct}
Can we establish explicit conditions on the number of distinct nonzero level eigenvalues for a rooted tree? Can the rooted trees having exactly $k$ distinct level eigenvalues be characterised for any $k>3$?
\end{question}
As a corollary to~\cite[Theorem~2.1]{Liu2015}, we can state the following result with regard to Question~\ref{Ques:distinct}.
\begin{corollary}
Let $T$ be a rooted tree with $n>1$ vertices and $\X$ the unit Perron vector of $L(T)$. For any integer $2\leq k \leq n$, $L(T)$ has $k$ distinct eigenvalues if and only if there exists $k-1$ pairwise distinct real numbers $\alpha_2,\alpha_3,\ldots,\alpha_k$ such that 
\begin{align*}
\rho_L(T) > \max_{2\leq j \leq k}\alpha_j \quad\text{and} \quad 
\prod_{2\leq j \leq k}(L(T)-\alpha_j I_n) = \prod_{2\leq j \leq k}(\rho_L(T)-\alpha_j)\X \X^T\,.
\end{align*}
Moreover,  $\rho_L(T), \alpha_2,\alpha_3,\ldots,\alpha_k$  are precisely the $k$ distinct eigenvalues of $L(T)$.
\end{corollary}

\medskip
The inverse problem of which real numbers can be realisable as eigenvalues of the level matrix of a rooted tree can also be studied. Since $L(S_n)$ has only two nonzero eigenvalues given by $\pm \sqrt{n-1}$, we see that the square root of any positive integer is a realisable level eigenvalue. In particular, since $E_L(T) = 2 \rho_L(T)$, we have that every even positive integer can be realised as the level energy of some rooted tree. Moreover, we know that the level eigenvalues (thus, the level energy) are all algebraic integers. It is well-known that if an algebraic integer is a rational number, then it must be an ordinary integer. The square root of any nonnegative integer $m$ is an algebraic integer, but an irrational number unless $m$ is a perfect square.

\begin{problem}\label{Prob:Inverse}
Explore the inverse problem of which real numbers can be realisable as eigenvalues of the level matrix of a rooted tree different from a rooted star. 
\end{problem}
 
\begin{question}
Given any rooted tree $T_1$, can we always construct from $T_1$ a new rooted tree $T_2$ whose level spectrum contains that of $T_1$?
\end{question}

Let $d$ be a positive integer. By a $d$-ary tree, we mean a rooted tree in which every vertex has precisely $d$ children.
\begin{question}
Are complete (full) $d$-ary trees characterised by their level spectra?
\end{question}

The computation of the characteristic polynomial of a tree is quite an old mathematical problem, and there are some known recursive formulae (for as eg. the adjacency and Laplacian matrices) that deal with it~\cite{Jacob,Mohar,Tinhofer}.
 \begin{question}
Can we establish a recursive formula that computes the level characteristic polynomial of any rooted tree?
\end{question}

Further research may include finding extremal trees for the level spectral radius among all rooted trees with a prescribed: (1) outdegree sequence, (2) number of vertices and number of leaves,  (3) number of vertices and maximum outdegree.

We hope to see more work related to the level matrix in the future.

\medskip
\section*{Acknowledgments}

The author thanks Dr.~Bunyamin~\c{S}ahin (Selcuk University, Turkey) for introducing him to the level matrix. This work is dedicated to Professor Stephan Wagner (Uppsala University, Sweden) on the occasion of his birthday.

\end{document}